\documentclass[leqno,12pt]{article} 
\usepackage{amsmath,amssymb}

\usepackage{amsmath, amssymb}
\usepackage{amsthm} 
\usepackage{amssymb}
\usepackage{mathrsfs}
\usepackage{amssymb, url, color, graphicx, amscd, mathrsfs}
\usepackage[colorlinks=true, bookmarks=true, pdfstartview=FitH, pagebackref=true, linktocpage=true, linkcolor = magenta, citecolor = blue]{hyperref}
\usepackage{graphicx}
\usepackage{times}
\theoremstyle{plain} 
\newtheorem{theorem}{\indent\sc Theorem}[section]
\newtheorem{lemma}[theorem]{\indent\sc Lemma}
\newtheorem{corollary}[theorem]{\indent\sc Corollary}

\theoremstyle{definition} 

\newtheorem{remark}[theorem]{\indent\sc Remark}

\numberwithin{equation}{section}

\def\({\left(}
\def\){\right)}

\newcommand{\KN}{\mathbin{\bigcirc\mspace{-15mu}\wedge\mspace{3mu}}}

\DeclareMathOperator{\Ric}{Ric}

\makeatletter
\def\address#1#2{\begingroup
\noindent\parbox[t]{8.7cm}{%
\small{\scshape\ignorespaces#1}\par\vskip1ex
\noindent\small{\itshape E-mail address}%
\/: #2\par\vskip4ex}\hfill%
\endgroup}%
\makeatother
\title{{Some results on weighted $p$-harmonic $\ell$-forms on weighted manifolds}} 
\author{
\textsc{Dang Tuyen Nguyen, Dhiriti Sundar Patra$^*$} 
}
\date{} 
\begin{document}

\maketitle

\footnote{ 
			$^{*}$ Corresponding author. 
	}
	\begin{abstract}
Basing on new curvature conditions for the Bochner
technique of Petersen-Wrink in \cite{PW21}, in this paper, we give some vanishing properties for 
weighted $L^Q$ $p$-harmonic $\ell$-forms on weighted manifolds, where $p>1, Q>p-1.$ These results generalize some previous theorems for $L^Q$ harmonic $\ell$-forms of Dung-Dung-Hung on the paper \cite{DDH25}.
	\end{abstract}
\section{Introduction}

Harmonic differential forms are differential forms $\omega$ satisfying
\[
\Delta\omega=(dd^{*}+d^{*}d)\omega=0,
\]
where $d$ denotes the exterior derivative, $d^{*}$ is its formal adjoint, and $\Delta$ is the Hodge--Laplacian. Hodge theory provides an important bridge connecting geometric and topological properties of manifolds. One of the key tools in studying harmonic differential forms is Bochner’s method; it provides links between the Hodge-Laplace operator and the curvature of the space and leads to many vanishing and rigidity theorems based on various curvature restrictions.
A recent important development in the Bochner technique came from Petersen and Wink~\cite{PW20,PW21,PW22}. They proposed new curvature conditions based on suitable sums of eigenvalues of the curvature operator. Their framework provides improved Bochner estimates and vanishing theorems with weaker curvature assumptions than those in the classical theory. This considerably extends the applicability of the Bochner method.  Later, Colombo, Mariani, and Rigoli~\cite{CMR24} built on this approach and established several new vanishing and rigidity results under these improved curvature conditions.

The study of $L^p$-harmonic forms has become an important tool for understanding the geometry and topology of complete non-compact manifolds. One of the earliest contributions came from Carron \cite{Car98,Car99}. He developed the theory of $L^2$-cohomology and proved finiteness results under Sobolev inequality and integral curvature assumptions. He also provided several geometric and topological applications. Additionally, Zhang \cite{ZG01} proved vanishing theorems for $p$-harmonic forms on complete Riemannian manifolds with nonnegative Ricci curvature. Later, Chang, Guo, and Sung \cite{CGS10} established compactness results for bounded $p$-harmonic $1$-forms. For more recent developments, we suggest that the readers refer to \cite{CY16,CHB21,DDT23}.

On the other hand, smooth metric measure spaces, also known as weighted manifolds, have become a popular area of research due to their connections with the Bakry–Émery Ricci tensor, Ricci solitons, and diffusion geometry (see \cite{LO03,WW09}).  A smooth metric measure space is a triple $(M^n,g,e^{-f}dv)$, where $(M^n,g)$ is a complete $n$-dimensional Riemannian manifold and $e^{-f}dv$ is the weighted volume measure, with $dv$ denoting the Riemannian volume element.
An $\ell$-form $\omega$ is called \emph{weighted $p$-harmonic} $(p>1)$ if
\[
d\omega=0,\qquad d_f^*(|\omega|^{p-2}\omega)=0,
\]
where $d_f^*$ denotes the adjoint of $d$ with respect to the weighted measure $e^{-f}dv$ (see Section \ref{sec2}). A weighted $p$-harmonic $\ell$-form $\omega$ on $(M^n,g,e^{-f}dv)$ is called an $L^Q$ weighted $p$-harmonic $\ell$-form  if
\[
\|\omega\|_{L_f^Q(M)}
=
\left(
\int_M |\omega|^Qe^{-f}dv
\right)^{\frac1Q}
<\infty.
\]
A significant amount of results on harmonic forms and $p$-harmonic forms have since been applied to smooth metric measure spaces. Bueler \cite{B99} began the study of the weighted Hodge Laplacian on noncompact manifolds. Li and Wang \cite{LW06} developed weighted Poincaré inequalities and produced rigidity results that have become valuable tools in the weighted context. Later, Vieira \cite{VI13} and Dung and Sung \cite{DS19} explored weighted $p$-harmonic forms, establishing several vanishing theorems and applications. More recently, Petersen and Wink \cite{PW20} introduced weighted curvature tensors and expanded their new Bochner technique to smooth metric measure spaces, providing the weighted curvature framework used in this paper.

Recently, building on the framework of Petersen and Wink, Dung-Dung-Hung \cite{DDH25} established several vanishing results for $L^Q$ weighted  harmonic $\ell$-forms on weighted manifolds under suitable weighted curvature assumptions. Motivated by this work, we study $L^Q$ weighted $p$-harmonic $\ell$-forms on $(M^n,g,e^{-f}dv)$, where $Q>p-1$ and $p>1$, thereby extending the weighted harmonic theory to the nonlinear setting. 
In order to state the results, we recall the tensor $h$ which is introduced in \cite{DDH25} 
	$$
h=-\frac{1}{n-2\ell}\operatorname{Hess}f
-\frac{\Delta f}{2(n-\ell)(n-2\ell)}g,
$$ where $1\le \ell < \frac{n}{2}$.
Denote $\lambda_1\le \cdots \le \lambda_{\binom{n}{2}}$ the eigenvalues 
of the weighted curvature tensor
$\mathrm{Rm}+h \KN g.$
 First, we state a vanishing result for the case of $Q=p>1$  as follows.
\begin{theorem}
	Let $M_f=(M^n,g,e^{-f}dv)$ be an $n$-dimensional complete non-compact weighted manifold of infinite weighted volume and $\omega$ be an $L^p$ weighted $p$-harmonic $\ell$-form on $M_f$, where $p>1,$ and $1\le \ell < \frac{n}{2}$. 
	If	$$\lambda_1+\cdots+\lambda_{n-\ell}\ge 0	$$
	then $\omega$ is trivial.
\end{theorem}
 More generally, for arbitrary numbers $Q>p-1,p>1$, we obtain 
\begin{theorem}\label{th1}
	Let $M_f=(M^n,g,e^{-f}dv)$ be an $n$-dimensional complete non-compact weighted manifold of infinite weighted volume and $\omega$ be an $L^Q$ weighted $p$-harmonic $\ell$-form on $M_f$, where $Q>p-1,p>1,$ and $1\le \ell < \frac{n}{2}$. 
	If	$$\lambda_1+\cdots+\lambda_{n-\ell}\ge 0	$$
	and $$(p-1)(1+Q-p)-|p-2||Q-p|>0$$
	then $\omega$ is trivial.
\end{theorem}
\begin{remark}
	If $p=2$ then the condition $(p-1)(1+Q-p)-|p-2||Q-p|>0$ holds for any $Q>1$. Therefore, Theorem \ref{th1} reduces to Theorem 1.1 in \cite{DDH25}.
\end{remark}

We would like to mention that if $\ell=1$ then the conditions on boundedness of the weighted function $f$ and non-negativity of $\lambda_1+\cdots+\lambda_{n-1}$ lead to $\operatorname{Vol}_f(M)=\infty$ (see \cite{DDH25}). Therefore, Theorem \ref{th1} gives 
\begin{corollary}
	Let $M_f=(M^n,g,e^{-f}dv)$ be an $n$-dimensional complete non-compact weighted manifold  and $\omega$ be an $L^Q$ weighted $p$-harmonic $1$-form on $M_f$, where $Q>p-1,p>1$. 
	If	$f$ is bounded, $$\lambda_1+\cdots+\lambda_{n-1}\ge 0	$$
	and $$(p-1)(1+Q-p)-|p-2||Q-p|>0$$
	then $\omega$ is trivial.
\end{corollary}
\begin{remark}
    This corollary generalizes and improves Corollary 3.1 in \cite{DDH25} and Corollary 4.1 in \cite{VI13}.
\end{remark}
Next, if a weighted Poincar\'e inequality holds true on $M$ then we can omit the assumption $\rm{Vol}_f(M)=\infty$ and $M$ can have negative weighted curvature. Recall that $M$ satisfies a weighted Poincar\'e inequality with a weighted function $\rho$ ($\rho$ is non-negative) if for any smooth function $\phi \in C^{\infty}_0(M)$, we have the following inequality
\begin{align} \label{Poi}
	\int_M \rho(x) \phi(x)e^{-f}dv \le \int_M |\nabla \phi|e^{-f}dv.
\end{align} 
Now we can give the next result as follows.
\begin{theorem}\label{th2}
	Let $M_f=(M^n,g,e^{-f}dv)$ be an $n$-dimensional complete non-compact weighted manifold satisfying the weighted Poincar\'e inequality \eqref{Poi} and $\omega$ be an $L^Q$ weighted $p$-harmonic $\ell$-form on $M_f$, where $Q>p-1,p>1,$ and $1\le \ell < \frac{n}{2}$. 
	If	$$\lambda_1+\cdots+\lambda_{n-\ell}\ge -\kappa(n-\ell)\rho(x)	$$
	for some constant $$0\le \kappa\le \frac{4(p-1)(1+Q-p)-4|p-2||Q-p|}{\ell (n-\ell)(p+|Q-p|)^2}$$
	then $\omega$ is trivial.
\end{theorem}
In particular, for the case $Q=p>1$, we have the following result for $L^p$ weighted $p$-harmonic $\ell$-forms.
\begin{theorem}
	Let $M_f=(M^n,g,e^{-f}dv)$ be an $n$-dimensional complete non-compact weighted manifold satisfying the weighted Poincar\'e inequality \eqref{Poi} and $\omega$ be an $L^p$ weighted $p$-harmonic $\ell$-form on $M_f$, where $p>1,$ and $1\le \ell < \frac{n}{2}$. 
	If	$$\lambda_1+\cdots+\lambda_{n-\ell}\ge -\kappa(n-\ell)\rho(x)	$$
	for some constant $$0\le \kappa\le \frac{4(p-1)}{\ell (n-\ell)p^2}$$
	then $\omega$ is trivial.
\end{theorem}
\begin{remark}
	Theorem \ref{th2} can be seen as a generalization of Theorem 1.3 in \cite{DDH25}.
\end{remark}
The weighted Sobolev inequality is a fundamental analytic tool on complete non-compact weighted manifolds and often provides a more flexible framework than the weighted Poincar\'e inequality in geometric analysis; see, for example, \cite{Hebey99,SaloffCoste02}. Motivated by recent vanishing results established under the weighted Poincar\'e inequality \cite{DS19,DDH25}, it is natural to investigate whether the same conclusions remain valid under the weaker assumption of a weighted Sobolev inequality:
\begin{align}\label{Sob}
	\left(
	\int_M \phi^{\frac{2n}{n-2}} e^{-f}\,dv
	\right)^{\frac{n-2}{n}}
	\le
	C_S \int_M |\nabla \phi|^2 e^{-f}\,dv
\end{align} 
where $C_S$ is a positive constant and $\phi$ is any function in $C^{\infty}_0(M)$. The following theorem provides such an extension and, in particular, recovers Theorem~1.3 of \cite{DDH25} when $p=2$. Consequently, the theorem broadens the applicability of the vanishing theory for weighted $p$-harmonic forms on smooth metric measure spaces.

\begin{theorem}\label{th3}
	Let $M_f=(M^n,g,e^{-f}dv)$ be an $n$-dimensional complete non-compact weighted manifold satisfying the weighted Sobolev inequality \eqref{Sob} and $\omega$ be an $L^Q$ weighted $p$-harmonic $\ell$-form on $M_f$, where $Q>p-1,p>1,$ and $1\le \ell < \frac{n}{2}$. 
	If	$$\lambda_1+\cdots+\lambda_{n-\ell}\ge -\kappa(n-\ell)\rho(x)	$$
	where $$\|\rho\|_{\frac{n}{2}}=\left(\int_M \rho^{\frac{n}{2}}e^{-f}dv \right)^{\frac{2}{n}}<\infty$$
	 and $$0\le \kappa\le \frac{4(p-1)(1+Q-p)-4|p-2||Q-p|}{\ell (n-\ell)C_S \|\rho\|_{\frac{n}{2}}(p+|Q-p|)^2}$$
	then $\omega$ is trivial.
\end{theorem}
As a consequence, for $Q=p>1$, Theorem \ref{th3} yields to
\begin{theorem}
	Let $M_f=(M^n,g,e^{-f}dv)$ be an $n$-dimensional complete non-compact weighted manifold satisfying the weighted Sobolev inequality \eqref{Sob} and $\omega$ be an $L^p$ weighted $p$-harmonic $\ell$-form on $M_f$, where $p>1,$ and $1\le \ell < \frac{n}{2}$. 
	If	$$\lambda_1+\cdots+\lambda_{n-\ell}\ge -\kappa(n-\ell)\rho(x)	$$
	where $$\|\rho\|_{\frac{n}{2}}=\left(\int_M \rho^{\frac{n}{2}}e^{-f}dv \right)^{\frac{2}{n}}<\infty$$
	 and $$0\le \kappa\le \frac{4(p-1)}{\ell (n-\ell)C_S \|\rho\|_{\frac{n}{2}}p^2}$$
	then $\omega$ is trivial.
\end{theorem}

The paper is organized as follows. In Section 2, we recall some useful background and lemmas. Finally, we use Section 3 to give the proofs for the main theorems.

\section{Preliminaries}\label{sec2}
Let $(V,g)$ be an $n$-dimensional Euclidean vector space, and let
$T^{(0,k)}(V)$ denote the space of $(0,k)$-tensors on $V$. We write
$\mathrm{Sym}^2(V)$ for the space of symmetric $(0,2)$-tensors. It is known that
\[
\mathrm{Sym}^2(\Lambda^2V)
=
\mathrm{Sym}_B^2(\Lambda^2V)
\oplus
\Lambda^4V,
\]
where $\mathrm{Sym}_B^2(\Lambda^2V)$ is the subspace consisting of all tensors
$R\in\mathrm{Sym}^2(\Lambda^2V)$ that satisfy the first Bianchi identity.
Any tensor in $\mathrm{Sym}_B^2(\Lambda^2V)$ is called an
\emph{algebraic curvature tensor}.
The corresponding algebraic curvature $(0,4)$-tensor is defined by
\[
\mathrm{Rm}(x,y,z,w)
=
R(x\wedge y,\; z\wedge w),
\qquad
\forall\, x,y,z,w\in V.
\]
For two symmetric tensors $S,T\in\mathrm{Sym}^2(V)$, their
Kulkarni--Nomizu product is given by
\[
\begin{aligned}
(S\KN T)(x,y,z,w)
={}&
S(x,z)T(y,w)-S(x,w)T(y,z) \\
&
+S(y,w)T(x,z)-S(y,z)T(x,w).
\end{aligned}
\]
In particular, the algebraic curvature tensor $I=\frac12\,g\KN g$
represents the curvature tensor of the unit sphere.
Moreover, every algebraic curvature tensor admits the orthogonal decomposition
\[
\mathrm{Rm} = \frac{\mathrm{Scal}}{2n(n-1)}\,g\KN g +
\frac{1}{n-2}\,g\KN\mathring{\mathrm{Ric}} + W,
\]
where
$\mathring{\mathrm{Ric}} =\mathrm{Ric} - \frac{\mathrm{Scal}}{n}g$
is the traceless Ricci tensor, and $W$ is the Weyl curvature tensor.

Let $C^\infty(\Lambda^\ell(M))$ denote the
space of smooth $\ell$-forms on $M$. The exterior derivative
$d:C^\infty(\Lambda^\ell(M))
\longrightarrow
C^\infty(\Lambda^{\ell+1}(M))$
has weighted formal adjoint $d_f^*$ with respect to the measure
$e^{-f}dv$, characterized by
\[
\int_M \langle d\beta,\alpha\rangle e^{-f}dv
=
\int_M \langle \beta,d_f^*\alpha\rangle e^{-f}dv,
\]
for all compactly supported smooth $(\ell-1)$-forms $\beta$ and
$\ell$-forms $\alpha$. It is given by
\[
d_f^*=d^*+i_{\nabla f},
\]
where $i_{\nabla f}$ denotes the interior product with the vector field
$\nabla f$.

Throughout this paper, we adopt the convention
\[
\Delta_f=d_f^*d+dd_f^*
\]
for the weighted Hodge Laplacian. 
The Hodge Laplacian satisfies the Weitzenb\"ock formula
\begin{equation}\label{eq:HodgeLaplacian}
    \Delta=\nabla^*\nabla+\operatorname{Ric},
\end{equation}
where $\nabla^*$ denotes the formal adjoint of the covariant derivative with respect to the Riemannian measure. The weighted formal adjoint of the covariant derivative is given by
\[
\nabla_f^*=\nabla^*+i_{\nabla f}.
\]
Moreover, using Cartan's identity
$L_{\nabla f}=di_{\nabla f}+i_{\nabla f}d,$
we obtain
\begin{equation}\label{eq:weighted_hodge}
\Delta_f
=d_f^*d+dd_f^*  
=d^*d+dd^*+i_{\nabla f}d+di_{\nabla f} 
=\Delta+L_{\nabla f},
\end{equation}
where $L_{\nabla f}$ denotes the Lie derivative along $\nabla f$.

\begin{lemma} [see \cite{DS19}]  \label{d}
	For any closed $\ell$-form $\omega$ and $\varphi \in \mathbf{C} ^\infty (M)$, we have
	$$|d(\varphi \omega)|=|d \varphi \wedge \omega| \le |d \varphi|. |\omega|.$$	
\end{lemma}

By using the Bochner formula and the assumption on the weighted curvature tensor, we obtain the following lemma, which plays an important role in the proof of the main results.
\begin{lemma}\label{Boc}
	Let $M_f=(M^n,g,e^{-f}dv)$ be an $n$-dimensional complete non-compact weighted manifold. For $1\le \ell <\frac{n}{2}, \kappa \ge 0, \rho(x)\ge 0$, if the eigenvalues $\lambda_1\le \cdots \le \lambda_{\binom{n}{2}}$  
	of the weighted curvature tensor
	$\mathrm{Rm}+h \KN g$
	satisfy $$\lambda_1+\cdots+\lambda_{n-\ell}\ge -\kappa(n-\ell)\rho(x)	$$
	then we have the following inequality
	\begin{align} \label{BocF}
	|\omega| \Delta_f |\omega|^{p-1} \le \langle d^*_f d (|\omega|^{p-2}\omega), \omega\rangle+\kappa \ell (n-\ell) \rho(x)|\omega|^p,
\end{align}
	for any weighted $p$-harmonic $\ell$-form $\omega$ on $M$.
\end{lemma}

\begin{proof}
	First, we recall the Bochner formula for the weighted Hodge-Laplacian
	$$\frac{1}{2}\Delta_f |\omega|^2=-|\nabla \omega|^2 + \langle \Delta_f \omega, \omega \rangle- \langle \Ric_f (\omega), \omega \rangle,$$
	where $\omega$ is any smooth $\ell$-form on $M$.
	
	Moreover, with the same assumption on eigenvalues, in the proof of Theorem 1.1 and Theorem 1.3 in \cite{DDH25}, the authors obtained 
	$$ \langle \Ric_f (\omega), \omega \rangle= \langle \Ric_{\mathrm{Rm}+h \KN g} (\omega), \omega \rangle \ge -\kappa \ell (n-\ell)\rho |\omega|^2 .$$
	Therefore, we have $$\frac{1}{2}\Delta_f |\omega|^2\le-|\nabla \omega|^2 + \langle \Delta_f \omega, \omega \rangle +\kappa \ell (n-\ell)\rho |\omega|^2.$$
	Next, by applying this estimate for the form $|\omega|^{p-2}\omega$, we see that
	$$\frac{1}{2}\Delta_f |\omega|^{2(p-1)}\le-|\nabla (|\omega|^{p-2}\omega)|^2 +  \langle \Delta_f (|\omega|^{p-2}\omega), |\omega|^{p-2}\omega \rangle +\kappa \ell (n-\ell)\rho |\omega|^{2(p-1)}.$$
	By using the definition $d^*_f(|\omega|^{p-2}\omega)=0$ and the following fact, $$\Delta_f(\psi ^2) = 2 \psi \Delta_f \psi -2|\nabla \psi|^2,$$ 
    the above inequality can be written as follows
	\begin{align*}\nonumber
		|\omega|^{p-1}\Delta_f|\omega|^{p-1} 
&	\le |\nabla |\omega|^{p-1}|^2{-} |\nabla(|\omega|^{p-2}\omega)|^2 + \langle \Delta_f (|\omega|^{p-2}\omega), |\omega|^{p-2}\omega \rangle +\kappa \ell (n{-}\ell)\rho |\omega|^{2(p-1)}\\\nonumber
&\le \langle d^*_f d (|\omega|^{p-2}\omega), |\omega|^{p-2}\omega \rangle +\kappa \ell (n-\ell)\rho |\omega|^{2(p-1)},
	\end{align*}
	where in the last step, we use the Kato inequality and the definition of $\Delta_f$.
	Hence, this inequality implies that
	\begin{align*}
	|\omega| \Delta_f |\omega|^{p-1} \le \langle d^*_f d (|\omega|^{p-2}\omega), \omega \rangle +\kappa \ell (n-\ell) \rho(x)|\omega|^p.
\end{align*}
The proof of Lemma \ref{Boc} is complete.
\end{proof}

\section{Proof of main results}

First, we give a proof of Theorem \ref{th1}. Recall that in Theorem \ref{th1}, we assume that $$\lambda_1+\cdots+\lambda_{n-\ell}\ge 0	.$$ Therefore, we can use inequality \eqref{BocF} in Lemma \ref{Boc}  for the case $\kappa=0$.

\begin{proof}[Proof of Theorem \ref{th1}]
	Using the non-negative weighted curvature tensor assumption and Lemma \ref{Boc}, we have
	\begin{align}\label{Boc1}
			|\omega| \Delta_f |\omega|^{p-1} \le  \langle d^*_f d (|\omega|^{p-2}\omega), \omega \rangle.
	\end{align}
	Now, we fix a point $o$ in $M$ and choose a cut-off smooth function with compact support $\varphi$ on $M$ which satisfies $|\nabla \varphi|\le \frac{2}{r}$ and
\begin{align}\label{func}
		\varphi(x)= \begin{cases}
		1 \quad \text{on} \quad B_o(r)\\
		0 \quad \text{on} \quad M\setminus B_o(2r).
	\end{cases}
\end{align}	
\textbf{Case 1: $Q \ge p$}.  For $q=Q-p\ge 0$,	using inequality \eqref{Boc1} and the cut-off function $\varphi$ in \eqref{func}, we obtain the integral estimate as following
	\begin{align} \label{c11}
		\int_M \varphi ^2|\omega |^{q+1} \Delta_f |\omega |^{p-1}e^{-f} dv \le 
		\int_M \langle d(|\omega|^{p-2}\omega),d(\varphi^2 |\omega|^q\omega) \rangle e^{-f}dv.		
	\end{align}
Using the Stokes theorem and noting that $\varphi$ has compact support, we obtain the following inequality for the left-hand side.
\begin{align}\label{del1}  \nonumber
	\int_M {{\varphi ^2}|\omega |^{q+1}\Delta_f |\omega |^{p-1}} e^{-f}dv 
	&=\int_M \langle \nabla |\omega|^{p-1}, \nabla ( \varphi^2 |\omega|^{q+1}) \rangle e^{-f}dv \nonumber \\ 
		&\ge  -2(p-1)\int_M {\varphi |\omega |^{p+q-1}  | \nabla \varphi| |\nabla |\omega | | } e^{-f}dv\nonumber \\& \quad+ 
	(p-1)(q+1) \int_M {{\varphi ^2}|\omega|^{p+q-2} |\nabla |\omega |{|^2}}e^{-f}dv,
\end{align}	
where we use the fact that $|\langle u, v \rangle |\le |u|.|v|$ in the last step.

Next, we estimate the right-hand side by using Lemma \ref{d} as follows
\begin{align} \label{d1} \nonumber
	\int_M \langle d(|\omega|^{p-2}\omega),d(\varphi^2 |\omega|^q\omega) \rangle e^{-f}dv\nonumber
&\le \int_M |d|\omega|^{p-2}||\omega||d(\varphi^2 |\omega|^q)||\omega| e^{-f}dv\\\nonumber
	&= \int_M |\nabla|\omega|^{p-2}||\omega||\nabla(\varphi^2 |\omega|^q)||\omega| e^{-f}dv\\ 
	&\le 2|p-2| \int_M \varphi |\omega|^{p+q-1}|\nabla \varphi||\nabla|\omega||e^{-f}dv \nonumber\\ & \quad+ 
	|p-2|q \int_M  \varphi^2 |\omega|^{p+q-2}|\nabla|\omega||^2 e^{-f}dv.
\end{align}
	Substituting two above inequalities into inequality \eqref{c11}, we have
	\begin{align}\nonumber
		A_1 \int_M \varphi^2 |\omega|^{Q-2}|\nabla|\omega||^2 e^{-f}dv \le B_1 \int_M \varphi |\omega|^{Q-1} |\nabla \varphi||\nabla|\omega||e^{-f}dv,
	\end{align}
	where we denote
	\begin{align*}
		&A_1=(p-1)(1+Q-p)-|p-2|(Q-p)\\
		&B_1=2(p-1)+2|p-2|.
	\end{align*}
	For any positive constant $\epsilon$, applying the Cauchy-Schwarz inequality for the right-hand side, the above inequality implies that
	\begin{align}\nonumber
		\left(A_1-\frac{\epsilon B_1}{2}\right) \int_M \varphi^2 |\omega|^{Q-2}|\nabla|\omega||^2 e^{-f}dv \le \frac{B_1}{2\epsilon} \int_M |\omega|^{Q} |\nabla \varphi|^2e^{-f}dv.
	\end{align}
	Note that since $A_1>0$, we can choose $\epsilon$ such that $A_1-\frac{\epsilon B_1}{2}>0$. Hence, by using the definition of $\varphi$, the above inequality leads to
		\begin{align}
	0\le	\left(A_1-\frac{\epsilon B_1}{2}\right) \int_{B_o(r)}  |\omega|^{Q-2}|\nabla|\omega||^2 e^{-f}dv \le \frac{2B_1}{\epsilon r^2} \int_{B_o(2r)} |\omega|^{Q} e^{-f}dv.
	\end{align}
	Now, by letting $r \to \infty$, this inequality deduces to $|\omega|$ is a constant. Combining with the assumption on infinite weighted volume of $M$, we conclude that $\omega =0$.
	
	\textbf{Case 2: $p-1<Q<p$}. Let $q=Q-p$ then $-1<q<0$. Then, inequality \eqref{Boc1} leads to
	\begin{align}\label{c12}
		\int\limits_M \varphi ^2|\omega |(|\omega|+\delta)^q \Delta_f |\omega |^{p-1} e^{-f} dv \le 
		\int_M \langle d(|\omega|^{p-2}\omega),d(\varphi^2 (|\omega|+\delta)^q\omega) \rangle  e^{-f}dv.
	\end{align}
	Next, we apply the Stokes theorem for the left-hand side and Lemma \ref{d} for the right-hand side to obtain the following estimates.
	\begin{align}\label{del2}\nonumber
		\int\limits_M {{\varphi ^2}|\omega |(|\omega|+\delta)^q\Delta_f |\omega |^{p-1}}e^{-f}dv \nonumber
			&=\int_M \langle \nabla |\omega|^{p-1}, \nabla ( \varphi^2 |\omega|(|\omega|+\delta)^q) \rangle e^{-f}dv\\ \nonumber
		&=   2(p{-}1)\int\limits_M {\varphi |\omega |^{p-1} (|\omega|+\delta)^q  \langle \nabla \varphi ,\nabla |\omega | \rangle }e^{-f}dv\\
         &\quad + (p{-}1) \int\limits_M {{\varphi ^2}|\omega|^{p-2} (|\omega|+\delta)^q |\nabla |\omega |{|^2}}e^{-f}dv  \nonumber\\ \nonumber
		&
		\quad+(p-1) q\int\limits_M {{\varphi ^2}|\omega|^{p-1} (|\omega|+\delta)^{q-1} |\nabla |\omega |{|^2}}e^{-f}dv\\\nonumber
		&\ge  -2(p-1)\int\limits_M {\varphi |\omega |^{p-1} (|\omega|+\delta)^q  | \nabla \varphi| |\nabla |\omega | | }e^{-f} dv \\ &\nonumber \quad  +
		(p-1) \int\limits_M {{\varphi ^2}|\omega|^{p-2} (|\omega|+\delta)^q |\nabla |\omega |{|^2}}e^{-f}dv\\ \nonumber
		& \quad
		+(p-1) q\int\limits_M {{\varphi ^2}|\omega|^{p-1} (|\omega|+\delta)^{q-1} |\nabla |\omega |{|^2}}e^{-f}dv\\ 
		&\ge  -2(p-1)\int\limits_M {\varphi |\omega |^{p-1} (|\omega|+\delta)^q  | \nabla \varphi| |\nabla |\omega | | }e^{-f} dv \notag \\ &\quad + 
		(p-1)(1+q) \int\limits_M {{\varphi ^2}|\omega|^{p-2} (|\omega|+\delta)^q |\nabla |\omega |{|^2}}e^{-f}dv,
	\end{align}
	and
	\begin{align}\label{d2}\nonumber
		&\int_M \langle d(|\omega|^{p-2}\omega),d(\varphi^2 (|\omega|+\delta)^q\omega) \rangle e^{-f}dv \\\nonumber
		&\le \int_M |d|\omega|^{p-2}||\omega||d(\varphi^2 (|\omega|+\delta)^q)||\omega|e^{-f}dv\\\nonumber
		&\le  \int_M |p-2||\omega|^{p-1} |\nabla |\omega|| \left[ 2\varphi (|\omega|+\delta)^q |\nabla \varphi|+|q| \varphi^2 (|\omega|+\delta)^{q-1} |\nabla |\omega||  \right] e^{-f} dv\\\nonumber
		&\le 2|p-2| \int_M \varphi |\omega|^{p-1} (|\omega|+\delta)^{q}  |\nabla \varphi||\nabla|\omega||e^{-f}dv\\
		& \quad + |p-2||q| \int_M  \varphi^2 |\omega|^{p-2} (|\omega|+\delta)^{q} |\nabla|\omega||^2 e^{-f}dv.
	\end{align}
	Here, we use the fact that $\varphi^2 |\omega|^{p-1} (|\omega|+\delta)^{q-1} |\nabla|\omega||^2\le \varphi^2 |\omega|^{p-2} (|\omega|+\delta)^{q} |\nabla|\omega||^2$ in the last steps of the above inequalities. Combining these estimates and \eqref{c12}, we see that
	\begin{align}\nonumber
		A_2 \int_M \varphi^2 |\omega|^{p-2} (|\omega|+\delta)^{Q-p}|\nabla|\omega||^2 e^{-f}dv \le B_2 \int_M \varphi |\omega|^{p-1} (|\omega|+\delta)^{Q-p}|\nabla \varphi||\nabla|\omega||e^{-f}dv,
	\end{align}
	where
	\begin{align*}
		&A_2=(p-1)(1+Q-p)-|p-2||Q-p|,\\
		&B_2=2(p-1)+2|p-2|.
	\end{align*}
	Therefore, for any $\epsilon>0$, applying the Cauchy-Schwarz inequality and noting that $(|\omega|+\delta)^{Q-p}\le |\omega|^{Q-p}$, we have
	\begin{align}\nonumber
		\left(A_2-\frac{\epsilon B_2}{2}\right) \int_M \varphi^2 |\omega|^{p-2}(|\omega|+\delta)^{Q-p}|\nabla|\omega||^2 e^{-f}dv &\le \frac{B_2}{2\epsilon} \int_M |\omega|^{p} (|\omega|+\delta)^{Q-p}|\nabla \varphi|^2e^{-f}dv\\ \nonumber
				&\le \frac{B_2}{2\epsilon} \int_M |\omega|^{Q} |\nabla \varphi|^2e^{-f}dv.
	\end{align}
Using the assumption $A_2>0$, we can use the same arguments as those in Case 1 to obtain $\omega=0$. Hence, the proof of Theorem \ref{th1} is complete.	
\end{proof}

Next, we prove Theorem \ref{th2}, where we assume the manifold satisfies the weighted Poincar\'e inequality \eqref{Poi} with the weighted function $\rho(x)$ and the eigenvalues satisfies $$\lambda_1+\cdots+\lambda_{n-\ell}\ge -\kappa(n-\ell)\rho(x).$$	
\begin{proof}[Proof of Theorem \ref{th2}]
	Recall that by Lemma \ref{Boc} and the curvature assumption, we have
	\begin{align} \label{Boc2}
		|\omega| \Delta_f |\omega|^{p-1} \le \langle d^*_f d (|\omega|^{p-2}\omega), \omega \rangle +\kappa \ell (n-\ell) \rho(x)|\omega|^p,
	\end{align}
	Similar to the proof of Theorem \ref{th1}, we consider two cases based on the range of $Q$.
	
	\textbf{Case 1: $Q \ge p$}.  For $q=Q-p\ge 0$,	inequality \eqref{Boc2} implies to
	\begin{align} \label{c21}\nonumber
		\int_M \varphi ^2|\omega |^{q+1} \Delta_f |\omega |^{p-1}e^{-f} dv &\le 
		\int_M \langle d(|\omega|^{p-2}\omega),d(\varphi^2 |\omega|^q\omega) \rangle e^{-f}dv\\
	&	\quad+\kappa\ell(n-\ell)\int_M \rho(x)\varphi^2 |\omega
		|^{Q}e^{-f}dv.		
	\end{align}
	In order to find an upper bound for the second term in the right-hand side, we use the weighted Poincar\'e inequality \eqref{Poi} as follows
	\begin{align*}
		\int_M \rho\varphi^2 |\omega
		|^{Q}e^{-f}dv&\le   \int_M \left| \nabla \left( \varphi |\omega|^{\frac{Q}{2}} \right) \right|^2 e^{-f}dv\\ \nonumber
		&=  \int_M \left[|\omega {|^{Q}}|\nabla \varphi {|^2}  + 
		\frac{Q^2 }{4}{\varphi ^2}|\omega|^{Q-2} |\nabla |\omega |{|^2}+ 
		Q {\varphi |\omega |^{Q-1} \langle \nabla \varphi ,\nabla |\omega | \rangle }\right]e^{-f}dv\\
		&\le  \int_M \left[|\omega {|^{Q}}|\nabla \varphi {|^2}  + 
		\frac{Q^2 }{4}{\varphi ^2}|\omega|^{Q-2} |\nabla |\omega |{|^2}+ 
		Q {\varphi |\omega |^{Q-1} | \nabla \varphi| |\nabla |\omega || }\right]e^{-f}dv.	
	\end{align*}
	Substituting this inequality and inequalities \eqref{del1}-\eqref{d1} into inequality \eqref{c21}, and letting 
		\begin{align*}
		&A_3=(p-1)(1+Q-p)-|p-2|(Q-p)-\frac{\kappa \ell (n-\ell)Q^2}{4}\\
		&B_3=2(p-1)+2|p-2|+\kappa \ell (n-\ell)Q,
	\end{align*}
	we obtain
	\begin{align}\nonumber
		A_3 \int_M \varphi^2 |\omega|^{Q-2}|\nabla|\omega||^2 e^{-f}dv &\le B_3 \int_M \varphi |\omega|^{Q-1} |\nabla \varphi||\nabla|\omega||e^{-f}dv\\ \nonumber
		&\quad+\kappa \ell (n-\ell) \int_M |\omega|^Q |\nabla \varphi|^2 e^{-f}dv.
	\end{align}
	For any positive constant $\epsilon$, the Cauchy-Schwarz inequality and the above estimate lead to
	\begin{align} \nonumber
		\left(A_3-\frac{B_3 \epsilon}{2}\right) \int_M \varphi^2 |\omega|^{Q-2}|\nabla|\omega||^2 e^{-f}dv \le \left(\kappa \ell (n-\ell)+\frac{B_3 }{2\epsilon} \right) \int_M |\omega|^Q |\nabla \varphi|^2 e^{-f}dv.
	\end{align}
	Note that if a weighted Poincar\'e inequality holds on $M$ then $\rm{Vol}_f(M)=\infty$ (see Lemma 3.4 in \cite{DDH25}).  Hence, since $A_3>0$, we can repeat the arguments in the proof of Theorem \ref{th1} to obtain $\omega$ is trivial.
	
	\textbf{Case 2: $p-1<Q<p$}. Denote by $q=Q-p$, then $-1<q<0$. Using the cut-off function defined in \eqref{func} and inequality \eqref{Boc2}, we see that
	\begin{align}\label{c22}\nonumber
		\int\limits_M \varphi ^2|\omega |(|\omega|+\delta)^q \Delta_f |\omega |^{p-1} e^{-f} dv &\le 
		\int_M \langle d(|\omega|^{p-2}\omega),d(\varphi^2 (|\omega|+\delta)^q\omega) \rangle e^{-f}dv\\
		&\quad +\kappa\ell(n-\ell)\int_M \rho(x)\varphi^2 |\omega|^{p}(|\omega|+\delta)^{q}e^{-f}dv.
	\end{align}
	Now, we estimate the second term in the right-hand side by using the weighted Poincar\'e inequality \eqref{Poi} as follows.
	\begin{align}\label{rh} \nonumber
		&\int\limits_M \rho\varphi ^2|\omega |^{p}(|\omega|{+}\delta)^q e^{-f}dv\\\nonumber
		&\le \int_M \left|\nabla \left( \varphi |\omega|^{\frac{p}{2}}(|\omega|{+}\delta)^{\frac{q}{2}} \right)\right|^2 e^{-f}dv\\\nonumber
		&=\int_M \left\{ |\omega|^p (|\omega|{+}\delta)^q |\nabla \varphi|^2 +\varphi^2 |\omega|^{p-2}(|\omega|{+}\delta)^{q-2} |\nabla |\omega||^2 \left[ \frac{p}{2}(|\omega|{+}\delta)+\frac{q}{2}|\omega|  \right]^2 \right\}e^{-f}dv\\\nonumber
		&\qquad + \int_M |\omega|^{p-1}\varphi (|\omega|{+}\delta)^{q-1} \langle \nabla \varphi, \nabla |\omega| \rangle \left[p(|\omega|{+}\delta)+q|\omega|\right]e^{-f}dv\\\nonumber
		&\le \int_M \left\{ |\omega|^p (|\omega|{+}\delta)^q |\nabla \varphi|^2 +\varphi^2 |\omega|^{p-2}(|\omega|{+}\delta)^{q-2} |\nabla |\omega||^2 \left[ \frac{p}{2}(|\omega|{+}\delta)+\frac{q}{2}|\omega|  \right]^2 \right\}e^{-f}dv\\\nonumber
		&\qquad+ \int_M \varphi|\omega|^{p-1} (|\omega|{+}\delta)^{q-1} |\nabla \varphi| |\nabla |\omega||  \left[p(|\omega|{+}\delta)+q|\omega|\right]e^{-f}dv\\\nonumber
		&\le \int_M \left\{ |\omega|^p (|\omega|{+}\delta)^q |\nabla \varphi|^2 +\frac{(p+|q|)^2}{4}\varphi^2 |\omega|^{p-2}(|\omega|{+}\delta)^{q} |\nabla |\omega||^2  \right\}e^{-f}dv\\
		&\qquad + (p+|q|)\int_M \varphi |\omega|^{p-1} (|\omega|{+}\delta)^{q} |\nabla \varphi| |\nabla |\omega|| e^{-f}dv.			
	\end{align}
	Here, in the last step, we use the fact that  $0< p(|\omega|+\delta)+q|\omega|  \le (p+|q|)(|\omega|+\delta)$ for the last two terms in the right-hand side.	Combining this estimate and inequalities \eqref{del2}-\eqref{d2}, and then letting
	\begin{align*}
		&A_4=(p-1)(1+Q-p)-|p-2||Q-p|-\frac{\kappa \ell (n-\ell)(p+|q|)^2}{4},\\
		&B_4=2(p-1)+2|p-2|+\kappa \ell (n-\ell)(p+|q|),
	\end{align*}
	we obtain the following inequality
	\begin{align*}\nonumber
		A_4 \int_M \varphi^2 |\omega|^{p-2} (|\omega|+\delta)^{Q-p}|\nabla|\omega||^2 e^{-f}dv & \le B_4 \int_M \varphi |\omega|^{p-1} (|\omega|+\delta)^{Q-p}|\nabla \varphi||\nabla|\omega||e^{-f}dv\\
		&\qquad +\kappa \ell (n-\ell)\int_M |\omega|^p (|\omega|+\delta)^q |\nabla \varphi|^2 e^{-f}dv.
	\end{align*}
	For a constant $\epsilon>0$, applying the Cauchy-Schwarz inequality for the first term in the right-hand side, the above estimate leads to
	\begin{align}\nonumber
		&\left(A_4-\frac{\epsilon B_4}{2}\right) \int_M \varphi^2 |\omega|^{p-2}(|\omega|+\delta)^{Q-p}|\nabla|\omega||^2 e^{-f}dv \\ \nonumber
		&\qquad\qquad\le \left( \frac{B_4}{2\epsilon}+\kappa \ell (n-\ell) \right)\int_M |\omega|^{p} (|\omega|+\delta)^{Q-p}|\nabla \varphi|^2e^{-f}dv\\ \nonumber
		&\qquad\qquad\le \left( \frac{B_4}{2\epsilon}+\kappa \ell (n-\ell) \right) \int_M |\omega|^{Q} |\nabla \varphi|^2e^{-f}dv.
	\end{align}
	Observe that  $\rm{Vol}_f(M)=\infty$. Therefore, similar to the proof of Theorem \ref{th1}, this estimate and the assumption $A_4>0$ yield that $\omega$ vanishes on $M$. We complete the proof of Theorem \ref{th2}.
\end{proof}
	
	Finally, we prove the vanishing result when $M$ satisfies a Sobolev inequality \eqref{Sob} and the eigenvalues satisfies $$\lambda_1+\cdots+\lambda_{n-\ell}\ge -\kappa(n-\ell)\rho(x).$$	
\begin{proof}[Proof of Theorem \ref{th3}]
	By the assumption on the eigenvalues, we can use the first part in the proof of Theorem \ref{th2}.
	
	\textbf{Case 1: $Q\ge p$}. Recall that in this case, we have inequality \eqref{c21}. Next, we apply the weighted Sobolev inequality \eqref{Sob} to estimate the last term in \eqref{c21} as follows.
	\begin{align*}
		&\int_M  \rho\varphi^2  |\omega|^{Q}e^{-f}dv\\
		&\le \left(\int_M \rho^{\frac{n}{2}}e^{-f}dv\right)^{\frac{2}{n}} \left[ \int_M \left(\varphi |\omega|^{\frac{Q}{2}}\right)^{\frac{2n}{n-2}}e^{-f}dv \right]^{\frac{n-2}{n}}\\
		&\le \|\rho\|_{\frac{n}{2}} C_S  \int_M \left|\nabla\left(\varphi |\omega|^{\frac{Q}{2}}\right)\right|^2 e^{-f}dv \\
		&	\le \|\rho\|_{\frac{n}{2}} C_S \int_M \left[|\omega {|^{Q}}|\nabla \varphi {|^2}  + 
		\frac{Q^2 }{4}{\varphi ^2}|\omega|^{Q-2} |\nabla |\omega |{|^2}+ 
		Q {\varphi |\omega |^{Q-1} | \nabla \varphi| |\nabla |\omega || }\right]e^{-f}dv,
	\end{align*}
	where we use the Holder inequality in the first step. This estimate and inequalities \eqref{del1}-\eqref{d1}, \eqref{c21} lead to
	\begin{align}\nonumber
		A_5 \int_M \varphi^2 |\omega|^{Q-2}|\nabla|\omega||^2 e^{-f}dv &\le B_5 \int_M \varphi |\omega|^{Q-1} |\nabla \varphi||\nabla|\omega||e^{-f}dv\\ \nonumber
		&\quad+\kappa \ell (n-\ell) C_s \|\rho\|_{\frac{n}{2}} \int_M |\omega|^Q |\nabla \varphi|^2 e^{-f}dv.
	\end{align}
	Here, we denote by 
	\begin{align*}
		&A_5=(p-1)(1+Q-p)-|p-2|(Q-p)-\frac{\kappa \ell (n-\ell)C_s \|\rho\|_{\frac{n}{2}} Q^2}{4}\\
		&B_5=2(p-1)+2|p-2|+\kappa \ell (n-\ell)C_s \|\rho\|_{\frac{n}{2}} Q.
	\end{align*}
	Hence, we have the following estimate
		\begin{align} \nonumber
		\left(A_5-\frac{B_5 \epsilon}{2}\right) \int_M \varphi^2 |\omega|^{Q-2}|\nabla|\omega||^2 e^{-f}dv \le \left(\kappa \ell (n-\ell)C_s \|\rho\|_{\frac{n}{2}}+\frac{B_5 }{2\epsilon} \right) \int_M |\omega|^Q |\nabla \varphi|^2 e^{-f}dv,
	\end{align}
	for any positive constant $\epsilon$.
	
	Moreover, by $M$ satisfies a weighted Sobolev inequality, we see that $M$ has infinite weighted volume. Then, similar to Theorem \ref{th2}, the above estimate and the assumption on $A_5>0$ deduce to $\omega=0$.
	
	\textbf{Case 2: $p-1<Q<p$.} For $q=Q-p$, we obtain inequality \eqref{c22}. Using the Holder inequality, the weighted Sobolev inequality, we have
	\begin{align*}
				&\int\limits_M \rho\varphi ^2|\omega |^{p}(|\omega|+\delta)^q e^{-f}dv\\
				&\le \left(\int_M \rho^{\frac{n}{2}}e^{-f}dv\right)^{\frac{2}{n}} \left[ \int_M \left(\varphi |\omega|^{\frac{p}{2}}(|\omega|+\delta)^{\frac{q}{2}}\right)^{\frac{2n}{n-2}}e^{-f}dv \right]^{\frac{n-2}{n}}\\
		&\le \|\rho\|_{\frac{n}{2}} C_S  \int_M \left|\nabla \left( \varphi |\omega|^{\frac{p}{2}}(|\omega|+\delta)^{\frac{q}{2}} \right)\right|^2 e^{-f}dv\\
			&\le \|\rho\|_{\frac{n}{2}} C_S  \int_M \left\{ |\omega|^p (|\omega|+\delta)^q |\nabla \varphi|^2 +\frac{(p+|q|)^2}{4}\varphi^2 |\omega|^{p-2}(|\omega|+\delta)^{q} |\nabla |\omega||^2  \right\}e^{-f}dv\\
			&\quad + \|\rho\|_{\frac{n}{2}} C_S(p+|q|)\int_M \varphi |\omega|^{p-1} (|\omega|+\delta)^{q} |\nabla \varphi| |\nabla |\omega|| e^{-f}dv.
	\end{align*}
	Therefore, by combining this inequality and inequalities \eqref{del2}-\eqref{d2}, \eqref{c22}, we obtain
		\begin{align*}\nonumber
		A_6 \int_M \varphi^2 |\omega|^{p-2} (|\omega|+\delta)^{Q-p}|\nabla|\omega||^2 e^{-f}dv \le &B_6 \int_M \varphi |\omega|^{p-1} (|\omega|+\delta)^{Q-p}|\nabla \varphi||\nabla|\omega||e^{-f}dv\\
		&\qquad +\kappa \ell (n-\ell)\|\rho\|_{\frac{n}{2}} C_S\int_M |\omega|^p (|\omega|+\delta)^q |\nabla \varphi|^2 e^{-f}dv,
	\end{align*}
	where
	\begin{align*}
		&A_6=(p-1)(1+Q-p)-|p-2||Q-p|-\frac{\kappa \ell (n-\ell)\|\rho\|_{\frac{n}{2}} C_S(p+|q|)^2}{4},\\
		&B_6=2(p-1)+2|p-2|+\kappa \ell (n-\ell)\|\rho\|_{\frac{n}{2}} C_S(p+|q|).
	\end{align*}
	This estimate leads to
		\begin{align}\nonumber
		&\left(A_6-\frac{\epsilon B_6}{2}\right) \int_M \varphi^2 |\omega|^{p-2}(|\omega|+\delta)^{Q-p}|\nabla|\omega||^2 e^{-f}dv \\ \nonumber
		&\qquad\qquad\le \left( \frac{B_6}{2\epsilon}+\kappa \ell (n-\ell) \|\rho\|_{\frac{n}{2}} C_S\right)\int_M |\omega|^{p} (|\omega|+\delta)^{Q-p}|\nabla \varphi|^2e^{-f}dv\\ \nonumber
		&\qquad\qquad\le \left( \frac{B_6}{2\epsilon}+\kappa \ell (n-\ell)\|\rho\|_{\frac{n}{2}} C_S \right) \int_M |\omega|^{Q} |\nabla \varphi|^2e^{-f}dv.
	\end{align}
	By the positivity of $ A_6$, similar to the above cases, this inequality implies the vanishing property of $\omega$. The proof of Theorem \ref{th3} is complete.
\end{proof}
	\vskip0.1cm
\noindent
{\bf Acknowledgment:}   We would like to express our thanks to Prof. Nguyen Thac Dung for many useful suggestions. The authors thank Sourav Nayak for fruitful discussions. A part of this paper was completed during a stay of the first
author  at Vietnam Institute for Advanced Study in Mathematics (VIASM). We would like to express our thanks to the staff there for the hospitality. The first author  was supported by NAFOSTED under grant number 101.02-2025.62.

\section*{\vskip-10mm\noindent Declarations}

\begin{itemize}

\item\vskip-1.5mm {Conflict of interest:} The authors declare no conflicts of interest.

\item\vskip-1.5mm {Ethics approval and consent to participate:} The work is original, not under consideration elsewhere, and approved by all authors.

\item\vskip-1.5mm {Author contribution:} All authors contributed equally to this work.

\item\vskip-1.5mm {Data availability statement:}
The manuscript has no associated data.
\end{itemize}

\address{ {\it  Dang Tuyen Nguyen}\\
	Department of Mathematics,\\ 
	Hanoi University of Civil Engineering,\\
	Hanoi, Vietnam.
}
{tuyennd@huce.edu.vn}

\address{{\it Dhriti Sundar Patra}\\
	Department of Mathematics,\\
	Indian Institute of Technology Hyderabad,\\
	Kandi-502284, Sangareddy
	Telangana, India.
}
{dhriti@math.iith.ac.in}
%

\end{document}